\documentclass[12pt]{amsart}
\usepackage{amsmath,amssymb,amsthm,mathtools}
\usepackage{mathrsfs,bm}
\usepackage{thmtools,thm-restate}
\usepackage{tikz-cd}
\usepackage{tikz}
\usepackage{enumitem}
\usepackage{graphicx}
\usetikzlibrary{matrix,arrows.meta,bending}
\usepackage{color}
\usepackage[numbers]{natbib}
\usepackage{booktabs}
\usepackage{longtable}
\usepackage{pgfplots}
\pgfplotsset{compat=1.14} 
\usetikzlibrary{arrows.meta,decorations.pathreplacing,positioning,calc} 
\usepackage[a4paper, left=3cm, right=3cm, top=3cm, bottom=3cm, footskip=15mm]{geometry}

\newtheorem{theorem}{Theorem}[section]

\newtheorem{lemma}[theorem]{Lemma}
\newtheorem{proposition}[theorem]{Proposition}
\newtheorem{corollary}[theorem]{Corollary}

\theoremstyle{definition}
\newtheorem{definition}[theorem]{Definition}

\theoremstyle{remark}
\newtheorem{remark}[theorem]{Remark}

\numberwithin{equation}{section}

\makeindex

\usepackage{fancyhdr}
\usepackage{calc} % ensures dimension arithmetic works robustly

\AtBeginDocument{%
  \setlength{\textwidth}{\dimexpr\paperwidth - 6cm\relax}%
  \setlength{\oddsidemargin}{\dimexpr 3cm - 1in\relax}%
  \setlength{\evensidemargin}{\dimexpr 3cm - 1in\relax}%
  \setlength{\marginparwidth}{2.5cm}%
}

\begin{document}

\title{Limited polynomials and Sendov's conjecture}

\author{T. Agama}
\address{Department of Mathematics, African Institute for Mathematical science, Ghana
}
%\curraddr{Department of Mathematics and Statistics,
%{Case Western Reserve University, Cleveland, Ohio 43403}
\email{theophilus@aims.edu.gh/emperordagama@yahoo.com}

%\author{Author Two}
%\address{Mathematical Research Section, School of Mathematical Sciences,
%Australian National University, Canberra ACT 2601, Australia}

%    General info

\date{\today}

\dedicatory{}


\begin{abstract}
 In this paper, we introduce and study a particular class of polynomials. We study the distribution of their zeros, including the zeros of their derivatives and the interaction between these two. We prove a weak variant of the Sendov conjecture in the case where the zeros are real and are of the same sign. 
\end{abstract}

\maketitle

\section{Introduction and motivation}

The location of critical points of polynomials relative to their zeros is a classical and still active area of complex analysis and polynomial geometry. The long-standing Sendov conjecture asserts that every complex polynomial $P_n$ of degree $n$ with all zeros in the closed unit disk has the property that each zero lies within unit distance of some critical point of $P_n'$. The conjecture, originally formulated by the Bulgarian mathematician Blagovest Sendov, has inspired a large body of partial results and refinements that illuminate different geometric and analytic mechanisms behind the interaction of zeros and critical points.\\

Progress on the Sendov conjecture has proceeded by proving the conjecture in numerous special regimes. For zeros very near the unit circle, early estimates were provided in the work of Miller. Finite-degree verifications were developed by Brown and collaborators, settling the conjecture for low degrees; see, in particular, the degrees $\leq 6$ and degree $\leq 8$ results. Substantial progress in the high-degree (asymptotic) direction was achieved by D\'egot, and most recently by Tao, who established the conjecture for polynomials of sufficiently large degree by combining delicate potential-theoretic and combinatorial arguments. Complementary structural approaches, which restrict the possible configurations of zeros, have been pursued by several authors, including Borcea; Together, these works sketch a landscape in which the conjecture is true in many regimes, but where flexible, general strategies for all configurations remain elusive. \cite{miller1993sendov,brown1991sendov,borcea1996sendov,brown1999proof,degot2014sendov,tao2022sendov}\\

The present paper approaches Sendov-type questions from a different angle by introducing and exploiting a natural global measure of a polynomial that encodes the geometric dispersion of its zeros. Concretely, for a polynomial $P_n(z)=\prod_{i=1}^n (z-a_i)$ we study the \emph{measure}
$$
\mathcal{M}(P_n):=\prod_{i=1}^n |a_i|,
$$
and we call $P_n$ \emph{$\epsilon$-limited} when $\mathcal{M}(P_n)<\epsilon$. The defining idea is elementary but potent: small product of moduli forces a strong multiplicative constraint on the zero set, and under natural additional hypotheses, this multiplicative constraint implies proximity relations between zeros and critical points. Intuitively, an $\epsilon$-limited polynomial cannot have all zeros simultaneously large in modulus; at least one zero must act as a ``small absorber,'' and this asymmetry is the key to controlling nearby critical points.\\

There are three central mechanisms behind our results.

\begin{enumerate}
\item \textbf{Local expansions of extremal zeros.} When one zero is significantly smaller (in modulus) than the others, it is convenient to expand the polynomial in powers of $(x-a_j)$ around that extremal zero. The resulting local expansion produces coefficients (``indices of expansion'') that are multiplicatively controlled by the product of the remaining zeros; in the $\epsilon$-limited setting these coefficients are small and give direct control of derivatives evaluated at the extremal zero.
\bigskip

\item \textbf{Derivative identities and combinatorial expressions.} The classical relation between derivative and elementary symmetric functions allows us to express $P'(x)$ (and higher derivatives) as finite symmetric sums (permutational products). Estimating these sums at the extremal zero and combining them with the local expansion bounds yields quantitative inequalities that constrain the possible locations of critical points.
\bigskip

\item \textbf{Squeezing by factorial growth.} Passing from coefficient bounds to bounds on derivatives uses factorial growth (via $k!$) and Stirling's formula. This amplifies the smallness of the expansion indices in strong control of higher derivatives; the amplification is what ultimately forces critical points to lie within prescribed distances of the extremal zero when $\epsilon$ is sufficiently small.
\end{enumerate}
\bigskip

Using these ideas, we obtain a family of results that can be viewed as weak but natural Sendov-type statements for polynomials whose zeros are confined to the positive real axis (or to the negative axis, by sign symmetry) and whose product of all but one zero is small. The main representative result (stated and proved in the sequel) shows that if $P$ is monic, all zeros are positive and $\frac{P(x)}{x-a_j}$ is $1$-limited where $a_j$ is the least zero, then every critical point of $P$ lies within distance $1$ of $a_j$. More generally, our estimates (Theorem \ref{real case} and Theorem \ref{basic inequality} in the sequel) quantify how shrinking $\epsilon$ forces higher-order derivatives to cluster near the extremal zero.

\medskip

\noindent\textbf{Organization of the paper.} Section 2 introduces limited polynomials, records basic closure properties (products, conjugation, scaling) and develops the language of local expansions and indices of expansion. Section 3 contains the analytic core: lemmas on local expansions, the derivative identities, and the proofs of the main theorems together with several corollaries that illustrate how the zeros of $P$ and of its derivatives become arbitrarily close as $\epsilon\to 0$. In Section 4, we discuss extensions and limitations of the method, including a roadmap toward adapting the arguments to complex zeros via modulus reduction and the obstacles that arise there. The paper closes with brief remarks on potential directions for strengthening the method and connecting it with existing potential-theoretic approaches to Sendov and related conjectures.

\section{Problem statement}

Let $P_n$ be a complex coefficient polynomial of degree $n$ and $P_n(a_i)=0$ for each $1\leq i\leq n$. The Sendov conjecture is the assertion that, if $|a_i|<1$ with $P_n(a_i)=0$, then there exists some $b_k$ with $P_n'(b_k)=0$ such that 
\begin{align}
|a_i-b_k|<1.\nonumber
\end{align} 

There has been and still is a flurry of research devoted to this problem, and manifestly the current literature contains dozens of papers just for the problem. There has really been substantive progress ever since it was posed. For example, it has been shown in \cite{miller1993sendov} that the conjecture holds for zeros near the unit circle. In \cite{brown1991sendov} the conjecture has been verified for the degree of at most six. This was further improved to polynomials of degree at most seven in \cite{borcea1996sendov} and polynomials of degree at most eight in \cite{brown1999proof}. The best result thus far concerning the Sendov conjecture is found in \cite{degot2014sendov}, where it is verified to hold for sufficiently large degree polynomials. This asymptotic version was further modified and established in a complete form in \cite{tao2022sendov}.\\

In the sequel, we will study a particular class of complex polynomials and it turns out that such polynomials ''almost'' satisfy the Sendov conjecture.

\section{Limited complex-valued polynomials}

In this section, we introduce the concept of limited polynomials. We study various settings in which polynomials of these forms are preserved.

\begin{definition}
Let $P_{n}(z)$ be a complex-valued polynomial of degree $n$ and let $\mathcal{Z}(P_n(z))=\{a_1,a_2,\ldots ,a_n\}$ be the set of zeros of $P_n(z)$. By the measure of $P_n(z)$, denoted $\mathcal{M}(P_n(z))$, we mean the value 
\begin{align}
\mathcal{M}(P_n(z))=\prod \limits_{i=1}^{n}|a_i|.\nonumber
\end{align} 
We say that $P_n(z)$ is $\epsilon$-limited if $\mathcal{M}(P_n(z))<\epsilon$, for some $\epsilon >0$, where $\epsilon =\mathrm{sup}\bigg(M(P_n(z))\bigg)$.
\end{definition}
\bigskip

The concept of a limited polynomial hinges essentially on the distribution of the zeros of a polynomial. This in some sense gives some information about the distribution of the zeros of their derivatives. In particular, the very notion that a polynomial is $1$-limited, for instance, is quite suggestive in it's own right. In such a case, we could suspect that all the zeros of the polynomial in question are small, or almost all the zeros are somewhat large, and the few exceptional ones are incredibly small with the tendency to absorb the large zeros upon taking their product. Next, we examine some elementary properties of the concept of limited polynomials. In the following section, we examine various settings under which this concept is preserved.

\subsection{Properties of limited polynomials}

In this section, we examine some properties of limited polynomials. We examine the various settings under which this concept is preserved.

\begin{proposition}\label{product}
Let $P(z)$ and $Q(z)$ be any $\epsilon$ and $\delta$ limited polynomials, respectively. If $\mathcal{Z}(P(z))\cap \mathcal{Z}(Q(z))=\emptyset$, where  $\mathcal{Z}(P(z))$ and $\mathcal{Z}(Q(z))$ are the set of zeros of $P(z)$ and $Q(z)$, respectively, then the product $P(z)Q(z)$ is $\epsilon \delta$-limited.
\end{proposition}

\begin{proof}
Specify two limited polynomials $P(z)$ and $Q(z)$, where $P(z)$ is $\epsilon$-limited and $Q(z)$ is $\delta$-limited with zeros $\mathcal{Z}(P(z))=\{a_1,a_2,\ldots a_n\}$ and $\mathcal{Z}(Q(z))=\{b_1,b_2\ldots, b_m\}$. It follows that 
\begin{align}
\prod \limits_{i=1}^{n}|a_i|<\epsilon \quad \mathrm{and}\quad \prod \limits_{j=1}^{m}|b_j|<\delta \nonumber
\end{align}
for some $\epsilon, \delta >0$, where $\epsilon =\mathrm{sup}\bigg(M(P(z)\bigg)$ and $\delta =\mathrm{sup}\bigg(M(Q(z)\bigg)$. Since the zeros of the product $P_nP_m$ are the union of the zeros of $P_n$ and $P_m$. That is $\mathcal{Z}(P(z)Q(z))=\mathcal{Z}(P(z))\cup \mathcal{Z}(Q(z)):=\{a_1,a_2,\ldots, a_n, b_1,b_2,\ldots b_m\}$. It follows that \begin{align}
\mathcal{M}(P(z)Q(z))&=\prod \limits_{i=1}^{n}|a_i|\prod \limits_{j=1}^{m}|b_j| \nonumber \\&<\epsilon \delta \nonumber
\end{align}
and the result follows immediately.
\end{proof}

\begin{proposition}
Let $\lambda \in \mathbb{C}$ and $P(x)$ be an $\epsilon$-limited polynomial. Then $\lambda P(x)$ is also $\epsilon$-limited.
\end{proposition}

\begin{proposition}
Let $P(x)=\prod \limits_{i=1}^{n}(x-a_i)$ be an $\epsilon$-limited polynomial. Then the polynomial $Q(x)=\prod \limits_{i=1}^{n}(x-\overline{a}_i)$ is also $\epsilon$-limited. In particular, $P(x)Q(x)$ is also $\epsilon^2$-limited polynomial.
\end{proposition}

\begin{proof}
The result follows by applying Proposition \ref{product}.
\end{proof}

\begin{proposition}
Let $P(x)=\prod \limits_{i=1}^{n}(x-a_i)$ be an $\epsilon$-limited polynomial of degree $n\geq 2$. If $a_i=\lambda_jb_j$, then $Q(x)=\prod \limits_{j=1}^{n}(x-b_j)$ is $\bigg(\frac{\epsilon}{\prod \limits_{j=1}^{n}|\lambda_j|}\bigg)$-limited.
\end{proposition}

\begin{proof}
Suppose that $P(x)=\prod \limits_{i=1}^{n}(x-a_i)$ is an $\epsilon$-limited polynomial, and let $a_i=\lambda_jb_j$. It follows that \begin{align}
\prod \limits_{j=1}^{n}|b_j|&=\frac{\prod \limits_{i=1}^{n}|a_i|}{\prod \limits_{j=1}^{n}|\lambda_j|}\nonumber \\&<\frac{\epsilon}{\prod \limits_{j=1}^{n}|\lambda_j|},\nonumber
\end{align}
since $P_n(x)$ is $\epsilon$-limited.
\end{proof}

\section{Critical points of limited polynomials}

In this section, we prove some preliminary results concerning the existence of zeros of limited polynomials with real zeros and the zeros of their derivatives, as well as the interplay between the two concerning their local and global distributions. We begin with the following elementary inequalities and identities. These two form a rich tool-box for further studies in the sequel.

\begin{lemma}\label{crilemma}
Let $P(x)=\prod \limits_{i=1}^{n}(x-a_i)$ be a monic polynomial of degree $n$, with $a_i\in \mathbb{R}^{+}$. There exists some $s_j$ for $j=1,2\ldots n-1$ such 
\begin{align}
\prod \limits_{i=1}^{n}(x-a_i)&=(x-a_j)^n+s_{n-1}(x-a_j)^{n-1}+s_{n-2}(x-a_j)^{n-2}+\cdots +s_{1}(x-a_j)\nonumber \\&=\sum \limits_{k=1}^{n}s_{k}(x-a_j)^k,\nonumber
\end{align}
where $a_j=\mathrm{min}\{a_i\}_{i=1}^{n}$ and $|s_j|<\prod \limits_{\substack{i=1\\ i\neq j}}^{n}|a_i|$.
\end{lemma}

\begin{proof}
Specify the monic polynomial 
\begin{align}
P(x)&=\prod \limits_{i=1}^{n}(x-a_i).\nonumber
\end{align}
Now choose $a_{j}=\mathrm{min}\{a_i\}_{i=1}^{n}$ and write \begin{align}
\prod \limits_{i=1}^{n}(x-a_i)&=(x-a_1)(x-a_2)\cdots (x-a_j)\cdots (x-a_n)\nonumber \\&=(x-a_j)(x-a_j-r_1)(x-a_j-r_2)\cdots (x-a_j-r_{n-2})(x-a_j-r_{n-1})\nonumber \\&=(x-a_j)((x-a_j)-r_1)((x-a_j)-r_2)\cdots ((x-a_j)-r_{n-2})((x-a_j)-r_{n-1}) \nonumber \\&=\sum \limits_{k=1}^{n}s_k(x-a_j)^k\nonumber
\end{align}
thereby establishing the relation.
\end{proof}

\begin{remark}
We prove an analog of the lemma \ref{crilemma}.
\end{remark}

\begin{lemma}\label{crilemma 2}
Let $P(x)=\prod \limits_{i=1}^{n}(x+a_i)$ be a monic polynomial of degree $n$, with $a_i\in \mathbb{R}^{+}$. There exists some $s_j$ for $j=1,2\ldots n-1$ such 
\begin{align}
\prod \limits_{i=1}^{n}(x+a_i)&=(x+a_j)^n+s_{n-1}(x+a_j)^{n-1}+s_{n-2}(x+a_j)^{n-2}+\cdots +s_{1}(x+a_j)\nonumber \\&=\sum \limits_{k=1}^{n}s_{k}(x+a_j)^k,\nonumber
\end{align}
where $a_j=\mathrm{max}\{a_i\}_{i=1}^{n}$ and $|s_j|<|a_j|^n$.
\end{lemma}

\begin{proof}
Specify the monic polynomial 
\begin{align}
P(x)&=\prod \limits_{i=1}^{n}(x+a_i).\nonumber
\end{align}
Now choose $a_{j}=\mathrm{max}\{a_i\}_{i=1}^{n}$ and write \begin{align}
\prod \limits_{i=1}^{n}(x+a_i)&=(x+a_1)(x+a_2)\cdots (x+a_j)\cdots (x+a_n)\nonumber \\&=(x+a_j)(x+a_j-r_1)(x+a_j-r_2)\cdots (x+a_j-r_{n-2})(x+a_j-r_{n-1})\nonumber \\&=(x+a_j)((x+a_j)-r_1)((x+a_j)-r_2)\cdots ((x+a_j)-r_{n-2})((x+a_j)-r_{n-1}) \nonumber \\&=\sum \limits_{k=1}^{n}s_k(x+a_j)^k\nonumber
\end{align}
thereby establishing the relation.
\end{proof}
\bigskip

Using the lemma \ref{crilemma}, we prove a weak variant of the Sendov conjecture.

\begin{definition}\label{crilemma 3}
Let $P(x):=\prod \limits_{i=1}^{n}(x-a_i)$ be a polynomial of degree $n$ with $a_i>0$. By the local expansion of $P(x)$, we mean the finite sum of the form 
\begin{align}
P(x)&=\sum \limits_{k=1}^{n}s_{k}(x-a_j)^k,\nonumber 
\end{align}
where $a_j=\mathrm{min}\{a_i\}_{i=1}^{n}$ and $|s_j|<\prod \limits_{\substack{i=1\\i\neq j}}^{n}|a_i|$. We call $s_j$ for $j=1,\ldots n-1$ the index of expansion.
\end{definition}

\begin{theorem}\label{real case}
Let $P(x)$ be a monic polynomial with real coefficients with index $|s_t|=\frac{1}{t}$ for $t=1\ldots, n-1$. Let $\mathcal{Z}(P(x))=\{a_1,a_2,\ldots, a_n\}\subset \mathbb{R}^{+}$ and $\mathcal{C}(P(x))=\{b_1,b_2,\ldots , b_{n-1}\}$ be the set of zeros and the set of critical points, respectively. If $\frac{P(x)}{x-a_j}$ is $1$-limited, then for each $b_i$ 
\begin{align}
|a_j-b_i|<1\nonumber
\end{align}
where $a_j=\mathrm{min}\{a_i\}_{i=1}^{n}$.
\end{theorem}

\begin{proof}
Specify the monic polynomial of degree $n\geq 2$ with real coefficients, given by 
\begin{align}
P(x)&=\prod \limits_{i=1}^{n}(x-a_i).\nonumber
\end{align}
By Lemma \ref{crilemma}, we have the decomposition 
\begin{align}
\prod \limits_{i=1}^{n}(x-a_i)&=\sum \limits_{k=1}^{n}s_{k}(x-a_j)^k,\nonumber
\end{align}
where $a_{j}=\mathrm{min}\{a_i\}_{i=1}^{n}$ and $|s_{j}|<\prod \limits_{\substack{i=1\\i\neq j}}^{n}|a_i|$. It follows that \begin{align}
\frac{dP(x)}{dx}&=\sum \limits_{k=1}^{n}ks_k(x-a_j)^{k-1}\nonumber \\&=\sum \limits_{k=2}^{n}ks_k(x-a_j)^{k-1}+s_1\nonumber \\&=(x-a_j)\bigg(\sum \limits_{k=2}^{n}ks_k(x-a_j)^{k-2}\bigg)+s_1.\nonumber
\end{align}
By setting $\frac{dP(x)}{dx}=0$, it follows that 
\begin{align}
\bigg|(x-a_j)\bigg(\sum \limits_{k=2}^{n}ks_{k}(x-a_j)^{k-2}\bigg)\bigg|&=|s_1|.\nonumber
\end{align}
We note that the values of $x$ that satisfy this relation are the zeros of the polynomial $P'_n(x)$. That is, for each $b_i\in \mathcal{C}(P(x))=\{b_1,b_2,\ldots,b_{n-1}\}$ the relation is valid \begin{align}
\bigg|(b_i-a_j)\bigg(\sum \limits_{k=2}^{n}ks_{k}(b_i-a_j)^{k-2}\bigg)\bigg|&=|s_1|.\nonumber
\end{align}
Since $\frac{P(x)}{x-a_j}$ is $1$-limited, it follows that $|s_j|<\prod \limits_{\substack{i=1\\ i\neq j}}^{n}|a_i|<1$ for each $j=1,\ldots n-1$ so that 
\begin{align}
|b_i-a_j|\bigg|\bigg(\sum \limits_{k=2}^{n}ks_k(b_i-a_j)^{k-2}\bigg)\bigg|&<1.\nonumber
\end{align}
Since $|s_k|=\frac{1}{k}$, it follows that 
\begin{align}
|b_i-a_j|<1.\nonumber
\end{align}
If we suppose $|b_i-a_j|\geq 1$, then we deduce
\begin{align}
1&<|b_i-a_j|\bigg|\bigg(\sum \limits_{k=2}^{n}ks_k(b_i-a_j)^{k-2}\bigg)\bigg|\nonumber \\&<1.\nonumber 
\end{align}
This inequality cannot hold. Thus, all the critical values are sufficiently close to the least zero of the polynomial.
\end{proof}
\bigskip

Theorem \ref{real case} can be considered as a weak variant of the result of the real case of the Sendov conjecture. The only compromise in this case is the assumption that all zeros are positive, in which case we have found that, indeed, the smallest zero of $P(x)$ must be close to all the zeros of the polynomial $P'(x)$. A similar approach could be adapted for the case where all zeros are negative. We will prove a result which, in some sense, exposes the distribution of the zeros of limited polynomials.

\begin{theorem}\label{basic inequality}
Let $P(x)=\prod \limits_{i=1}^{n}(x-a_i)$ be a polynomial of degree $n$ for $n\geq 2$, with $a_i\in \mathbb{R}^{+}$ and $a_i>0$. If $\frac{P(x)}{x-a_j}$ is $\epsilon$-limited, where $a_j=\mathrm{min}\{a_i\}_{i=1}^{n}$, then 
\begin{align}
\sum \limits_{s=1}^{n}\bigg|\frac{d^sP(a_j)}{dx^s}\bigg|&<\epsilon \sqrt{2\pi}\sum \limits_{k=1}^{n}e^{-k}k^{k+\frac{1}{2}}.\nonumber
\end{align}
\end{theorem}

\begin{proof}
Let $P(x)=\prod \limits_{i=1}^{n}(x-a_i)$ be a polynomial of degree $n$, with $a_i>0$, then by definition \ref{crilemma 3}, we can write 
\begin{align}
P(x)&=\sum \limits_{k=1}^{n}s_{k}(x-a_j)^k,\nonumber 
\end{align}
where $a_j=\mathrm{min}\{a_i\}_{i=1}^{n}$ and $|s_j|<\prod \limits_{\substack{i=1\\i\neq j}}^{n}|a_i|$. It follows that \begin{align}
\frac{dP(x)}{dx}&=\sum \limits_{k=1}^{n}ks_k(x-a_j)^{k-1}\nonumber \\&=\sum \limits_{k=2}^{n}ks_k(x-a_j)^{k-1}+s_1.\nonumber
\end{align}
It follows that $\frac{dP(a_j)}{dx}=s_1$. Again, by taking the second derivative $P''(x)$, we have 
\begin{align}
\frac{d^2P(x)}{dx^2}&=\sum \limits_{k=2}^{n}k(k-1)s_k(x-a_j)^{k-2}\nonumber \\&=\sum \limits_{k=3}^{n}k(k-1)s_k(x-a_j)^{k-2}+2s_2.\nonumber
\end{align}
Thus, it follows that $\frac{d^2P(a_j)}{dx^2}=2s_2$. By induction, it follows that 
\begin{align}
\sum \limits_{s=1}^{n}\bigg|\frac{d^sP(a_j)}{dx^s}\bigg|&=\sum \limits_{k=1}^{n}|k!s_k|\nonumber \\&\leq \sum \limits_{k=1}^{n}k!|s_k|\nonumber \\&<\epsilon \sum \limits_{k=1}^{n}k!\nonumber
\end{align}
since $\frac{P(x)}{x-a_j}$ is $\epsilon$-limited. Applying the Stirling formula, we can write 
\begin{align}
\sum \limits_{s=1}^{n}\bigg|\frac{d^sP(a_j)}{dx^s}\bigg|&<\epsilon \sqrt{2\pi}\sum \limits_{k=1}^{n}e^{-k}k^{k+\frac{1}{2}}.\nonumber
\end{align}
The result follows immediately from the relation.
\end{proof}
\bigskip

The inequality above highlights the relationship between the zeros of a given polynomial $P(x)$ and the higher order derivatives  $P^k(x)$ for $k=1,\ldots n$ in terms of the distribution. Indeed, for any fixed $n$, taking $\epsilon>0$ to be sufficiently small, the least zeros of $P(x)$ are sufficiently close to the zeros of each $P^k(x)$. Conversely, If we take $\epsilon>0$ somewhat large for a fixed $n$, then the least zero of $P(x)$ is far from at least one of the zeros of $P'(x)$. This property is archetypal of a very rare class of polynomials of which limited polynomials is a sub-class. The ensuing result is a strengthening of Theorem \ref{real case}. 

\begin{corollary}\label{squeeze}
Let $P(x)=\prod \limits_{i=1}^{n}(x-a_i)$ be a polynomial of degree $n$ for $n\geq 2$, with $a_i>0$ and let $\epsilon>0$ be arbitrary. If $\frac{P_n(x)}{x-a_j}$ is $\epsilon$-limited, where $a_j=\mathrm{min}\{a_i\}_{i=1}^{n}$, then $|a_j-b_{ik}|<\delta$, for any $\delta>0$ and $P^{k}(b_{ik})=0$~$(i=1,\ldots, n)$.
\end{corollary}

\begin{proof}
The result follows by varying the magnitude of $\epsilon$ in Theorem \ref{basic inequality}, since $\epsilon$ is arbitrary.
\end{proof}
\bigskip

The corollary \ref{squeeze} reveals quite well that the zeros of any real coefficient $\epsilon$-limited polynomial $P_n(x)$ can be made close to the zeros of their derivative $P'_n(x)$ by shrinking the size of $\epsilon$. In this case, we can set
\begin{align}
\epsilon&=\frac{1}{\sqrt{2\pi}\sum \limits_{k=1}^{n}e^{-k}(k+1)^{k+1}}.\nonumber
\end{align}
Similarly, the zeros of $P_n(x)$ can be far from the zeros of $P'_n(x)$ by taking $\epsilon>0$ to be smaller than 
\begin{align}
\sqrt{2\pi}\sum \limits_{k=1}^{n}e^{-k}k^{k+\frac{1}{2}}.\nonumber
\end{align}
Carefully tweaking $\epsilon$ this way gives some information about the local distribution of the zeros of the limited polynomials and the derivatives. 

\begin{corollary}\label{inequality 2}
Let $P(x)=\prod \limits_{i=1}^{n}(x-a_i)$ be a polynomial of degree $n$ for $n\geq 2$, with $a_i>0$. If $\frac{P(x)}{x-a_j}$ is $\epsilon$-limited, where $a_j=\mathrm{min}\{a_i\}_{i=1}^{n}$, then \begin{align}
\bigg|\sum \limits_{\sigma:[1,n]\longrightarrow [1,n]}\prod \limits_{\substack{\sigma(i)\neq \sigma(k)}}(a_j-a_{\sigma(i)})_{n-1}\bigg|<\epsilon \frac{\sqrt{2\pi}}{e},\nonumber
\end{align}
where $\prod \limits_{i\in [1,n]}(a_i)_k$ denotes the product of $k$ terms with index belonging to the set $[1,n]$.
\end{corollary}

\begin{proof}
Let $P(x)=\prod \limits_{i=1}^{n}(x-a_i)$ be a polynomial of degree $n$ for $n\geq 2$, with $a_i>0$. We observe that we can write \begin{align}
P'(x)&=\sum \limits_{\sigma:[1,n]\longrightarrow [1,n]}\prod \limits_{\sigma(i)\neq \sigma(k)}(x-a_{\sigma(i)})_{n-1},\nonumber
\end{align}
where the sum runs over all permutations. Taking $n=1$ in Theorem \ref{basic inequality}, we observe, on the other hand, that \begin{align}
\frac{dP(x)}{dx}&<\epsilon \frac{\sqrt{2\pi}}{e}.\nonumber
\end{align}
It follows that 
\begin{align}
\sum \limits_{\sigma:[1,n]\longrightarrow [1,n]}\prod \limits_{\sigma(i)\neq \sigma(k)}(a_j-a_{\sigma(i)})_{n-1}&<\epsilon \frac{\sqrt{2\pi}}{e}\nonumber
\end{align}
and the result follows immediately.
\end{proof}

\begin{remark}
The corollary \ref{inequality 2} is quite suggestive. Roughly speaking, it tells us that once a polynomial is $\epsilon$-limited for any $\epsilon>0$, the zeros should be at most within $\epsilon$ distance from the least zero. 
\end{remark}

\begin{corollary}\label{inequality 3}
Let $P(x)=\prod \limits_{i=1}^{n}(x-a_i)$ be a polynomial of degree $n$ for $n\geq 2$, with $a_i>0$. If $\frac{P(x)}{x-a_j}$ is $\epsilon$-limited, where $a_j=\mathrm{min}\{a_i\}_{i=1}^{n}$, then
\begin{align}
\sum \limits_{s=0}^{n-1}\bigg|\frac{d^s}{dx^s}\bigg(\sum \limits_{\sigma:[1,n]\longrightarrow [1,n]}\prod \limits_{\substack{\sigma(i)\neq \sigma(k)}}(x-a_{\sigma(i)})_{n-1} \bigg)(a_j)\bigg|<\epsilon \sqrt{2\pi}\sum \limits_{k=1}^{n}e^{-k}k^{k+\frac{1}{2}},\nonumber
\end{align}
where $\prod \limits_{i\in [1,n]}(a_i)_k$ denotes the product of $k$ terms with index belonging to the set $[1,n]$.
\end{corollary}

\begin{proof}
Let $P(x)=\prod \limits_{i=1}^{n}(x-a_i)$ for $n\geq 2$ with $a_i>0$. We can write 
\begin{align}
P'(x)&=\sum \limits_{\sigma:[1,n]\longrightarrow [1,n]}\prod \limits_{\sigma(i)\neq \sigma(k)}(x-a_{\sigma(i)})_{n-1}.\nonumber
\end{align}
The result follows by applying Theorem \ref{basic inequality}.
\end{proof}
\bigskip

\section{Further remarks}

In this section, we discuss how this method could in principle be used to give an analog proof of the result in Theorem \ref{real case} in the case where the polynomial in question is allowed to take at least some zeros from the complex plane $\mathbb{C}$. Given any complex coefficient polynomial of the form 
\begin{align}
P(z):=\prod_{i=1}^{n}(z-a_i),\nonumber
\end{align}
the first step is to basically convert any such complex coefficient polynomial to a polynomial with a real coefficient of the form \begin{align}
R(x)=\prod_{i=1}^{n}(x-|a_i|).\nonumber
\end{align}
Let us assume that the polynomial $R(x)$ satisfies the requirements of Theorem \ref{real case}. Applying Theorem \ref{real case}, we can then get information about the distribution of zeros and the critical points of the polynomial $R(x)$. Given that the function 
\begin{align}
|\cdot |:\mathbb{R}^{+}\times \mathbb{R}^{+}\subset \mathbb{C}\longrightarrow \mathbb{R}^{+}\nonumber
\end{align}
is bijective, we can return this information to the complex coefficient polynomial $P(z)$ observing that the zero $|a_j|$ corresponds to the zero $a_j\in \mathbb{R}^{+}\times \mathbb{R}^{+}$ of the polynomial $P_n(x)$ and the number $c_i\in \mathbb{R}^{+}$ which is the zero of $R'(x)$ corresponds to some number $b \in \mathbb{R}^{+}\times \mathbb{R}^{+}\subset \mathbb{C}$ such that $\left |c_i-|a_j|\right|<1$. A really good inverse theorem will suffice to transfer this information to an upper bound
\begin{align}
|b-a_j|<1\nonumber
\end{align}
and show that $b$ satisfies $P'(b)=0$.
\bigskip

%%%%%%%%%%%%%%%%%%%%%%%%%%%%%%%%%%%%%%%%%%%%%%%%%%%%%%%%%%%%%%%%%%%%%%%%
\footnote{
\par
.}%
%%%%%%%%%%%%%%%%%%%%%%%%%%%%%%%%%%%%%%%%%%%%%%%%%%%%%%%%%%%%%%%%%%%%%%%%

\bibliographystyle{amsplain}

\begin{thebibliography}{10}

\bibitem {brown1991sendov} J.E. Brown, \textit{On the Sendov conjecture for sixth degree polynomials.}, Proceedings of the American mathematical Society, vol. 175:4, 1991, pp. 939--946.\\

\bibitem {borcea1996sendov} L. Borcea, \textit{The Sendov conjecture for polynomials with at most seven distinct zeros}, Analysis, vol.16:2,  Oldenbourge wissenschaftsverlag, 1996, 137--160.\\

\bibitem {degot2014sendov} J. D{\'e}got, \textit{Sendov conjecture for high degree polynomials}, Proceedings of the American Mathematical Society,  vol.142:4 , 2014, pp.1337--1349.\\


\bibitem {brown1999proof} J.E. Brown and G. Xiang, \textit{Proof of the Sendov conjecture for polynomials of degree at most eight}, Journal of mathematical analysis and applications, vol.232:2,  Elsevier, 1991, 272--292.\\

\bibitem {miller1993sendov} M.J. Miller, \textit{On Sendov's conjecture for roots near the unit circle}, Journal of mathematical analysis and applications, vol.175, academic press Inc ,1993,  pp. 632--632.\\

\bibitem {tao2022sendov} T. Tao, \textit{Sendov’s conjecture for sufficiently-high-degree polynomials}, Acta Mathematica, vol. 229:2, Lehigh University Bethlehem, Penn., USA, 2022, 347--392.

\end{thebibliography}

\end{document}